\documentclass[11pt]{article}
\usepackage{amssymb,amsmath,amsthm,amsfonts}
\usepackage{graphicx}
\usepackage{epsfig}

\usepackage{amssymb,amsmath,amsthm,amsfonts,mathrsfs, bbm,dsfont}

\theoremstyle{plain}

\theoremstyle{definition}
\newtheorem{definition}{Definition}[section]
\newtheorem{example}{Example}[section]
\newtheorem{remark}{Remark}[section]

\newtheorem{theorem}{Theorem}[section]
\newtheorem{lemma}{Lemma}[section]
\newtheorem{proposition}{Proposition}[section]

\numberwithin{equation}{section}

\begin{document}
\openup 0.8\jot
\title{\Large\bf The construction of observable algebra in field algebra of $G$-spin models determined by a normal subgroup \thanks{This work is supported
by National Science Foundation of China (10971011,11371222)} }
\author{ Xin Qiaoling, Jiang Lining \thanks{E-mail address: jianglining@bit.edu.cn}}
\date{}
\maketitle\begin{center}
\begin{minipage}{16cm}
{\small \it School of Mathematics and Statistics, Beijing Institute
of Technology, Beijing 100081, P. R. China}
\end{minipage}
\end{center}
\vspace{0.05cm}
\begin{center}
\begin{minipage}{16cm}
{\small {\bf Abstract}: Let $G$ be a finite group and $H$ a normal subgroup. Starting from $G$-spin models, in which a non-Abelian field ${\mathcal{F}}_H$ w.r.t. $H$ carries an action of the Hopf $C^*$-algebra $D(H;G)$, a subalgebra of the quantum double $D(G)$, the concrete construction of the observable algebra ${\mathcal{A}}_{(H,G)}$ is given, as $D(H;G)$-invariant subspace. Furthermore, using the iterated twisted tensor product, one can prove that the observable algebra ${\mathcal{A}}_{(H,G)}=\cdots\rtimes H\rtimes\widehat{G}\rtimes H\rtimes\widehat{G}\rtimes H\rtimes\cdots$, where $\widehat{G}$ denotes the algebra of complex functions on $G$, and $H$ the group algebra.
}
\endabstract
\end{minipage}\vspace{0.10cm}
\begin{minipage}{16cm}
\end{minipage}
\end{center}
\begin{center} \vspace{0.01cm}
\end{center}

\section{ Introduction}
 Let $G$ be a finite group with a unit $e$. The $G$-valued spin configuration on the 2-dimensional square lattices is the map $\sigma\colon{\Bbb Z}^2\rightarrow G$ with Euclidean action functional
 \begin{eqnarray*}
    \begin{array}{c}
 S(\sigma)=\sum \limits_{(x,y)} f(\sigma_x^{-1}\sigma_y),
 \end{array}
\end{eqnarray*}
where the summation runs over the nearest neighbor pairs in ${\Bbb Z}^2$ and $f\colon G\rightarrow {\Bbb R}$ is a function of the positive type. This kind of classical statistical systems or the corresponding quantum field theories are called $G$-spin models, see \cite{S.Dop,V.F.R.J, K.Szl}. Such models provide the simplest examples of lattice field theories exhibiting quantum symmetry. Generally, $G$-spin models with an Abelian group $G$ have a symmetry structure of $G\times \widetilde{G}$, where $\widetilde{G}$ is the Pontryagin dual of $G$. If $G$ is non-Abelian, the Pontryagin dual loses its meaning, and one usually considers the quantum double $D(G)$ of $G$, {\cite{K.A.Da,G.Mas}}. Here $D(G)$ is defined as the crossed product of $C(G)$, the algebra of complex functions on $G$, and group algebra ${\Bbb C}G$ with respect to the adjoint action of the latter on the former. Then $D(G)$ becomes a Hopf *-algebra of finite dimension \cite{P.Ban,C.Kas,F.Nil}. As in the traditional quantum field theory, one can define a field algebra ${\mathcal{F}}$ associated with this models, which is a $C^*$-algebra generated by
$\{\delta_g(x), \rho_h(l)\colon g\in G, h\in G, x\in {\Bbb Z}, l\in {\Bbb Z}+\frac{1}{2}\}$ subject to some relations \cite{K.Szl}. There is a natural action of $D(G)$ on ${\mathcal{F}}$ so that ${\mathcal{F}}$ becomes a $D(G)$-module algebra. Under this action on ${\mathcal{F}}$, the observable algebra ${\mathcal{A}}_G$ is obtained. In \cite{F.Nil}, Nill and Szlach$\mathrm{\acute{a}}$nyi pointed out ${\mathcal{A}}_G=\cdots \rtimes G\rtimes \widehat{G}\rtimes G\rtimes \widehat{G}\rtimes G\rtimes \cdots$, where the crossed product is taken with respect to the natural left action of the latter factor on the former one. In this paper, we extend the result to a general situation.

Assume that $G$ is a finite group and $H$ is a normal subgroup of $G$. In our previous paper \cite{Qiao},
we define a Hopf $C^*$-algebra $D(H;G)$, which is only a subalgebra of $D(G)$. Subsequently, we construct an algebra $\mathcal{F}_H$ in
the field algebra $\mathcal{F}$ of $G$-spin models, which is a $C^*$-algebra generated by $\{\delta_g(x), \rho_h(l)\colon g\in G, h\in H; x\in {\Bbb Z}, l\in {\Bbb Z}+\frac{1}{2}\}$, called the field algebra of $G$-spin models determined by $H$. There also exists a
natural action of $D(H;G)$ on $\mathcal{F}_H$, such that $\mathcal{F}_H$ is a $D(H;G)$-module algebra whereas F is not. Then
the observable algebra ${\mathcal{A}}_{(H,G)}$, which is the set of fixed points of $\mathcal{F}_H$ under the action of $D(H;G)$ is
obtained. In Section 2, we point out the concrete construction of the observable algebra ${\mathcal{A}}_{(H,G)}$.

In Section 3, we identify $H$ with the group algebra $\Bbb{C}H$, and $\widehat{G}$ the set of complex functions on $G$.
We firstly construct iterated twisted tensor product algebras of three factors, i.e. $H\bigotimes_{R_{0,1}}\widehat{G}\bigotimes_{R_{1,2}} H$ and $\widehat{G}\bigotimes_{R_{1,2}} H\bigotimes_{R_{0,1}}\widehat{G}$ by means of twisting maps $R_{0,1}$ and $R_{1,2}$, and then by induction build an iterated twisted product of any number of factors
\begin{eqnarray*}
    \begin{array}{c}
A_{n,m}=A_n\bigotimes_{R_{n,n+1}}A_{n+1}\bigotimes_{R_{n+1,n+2}}\cdots A_{m-1}
\bigotimes_{R_{m-1,m}}A_m,
\end{array}
\end{eqnarray*}
where $n,m\in \Bbb{Z}$ with $n<m$ and
$A_i=\left\{\begin{array}{cc}
                          H, & {\mathrm{if}} \ i \ {\mathrm{is}} \ {\mathrm{even}}\\
                          \widehat{G}, & {\mathrm{if}} \ i \ {\mathrm{is}} \ {\mathrm{odd}}
                        \end{array}\right.$,
which is a $C^*$-algebra of finite dimension. Let $n<n'$ and $m'<m$, one can show that $A_{n',m'}\subseteq A_{n,m}$, and then by the $C^*$-inductive limit of $A_{n,m}$,
 we obtain
 \begin{eqnarray*}
    \begin{array}{c}
    {\mathcal{A}}={(C^*)}\lim\limits_{n<m} A_{n,m},
\end{array}
\end{eqnarray*}
that is,
\begin{eqnarray*}
    \begin{array}{c}
    {\mathcal{A}}=\cdots H\bigotimes_{R_{0,1}} {\widehat{G}}\bigotimes_{R_{1,2}} H\bigotimes_{R_{0,1}}{\widehat{G}}\bigotimes_{R_{1,2}} H\bigotimes_{R_{0,1}}{\widehat{G}}\cdots.
\end{array}
\end{eqnarray*}
 Finally, we prove that there is a $C^*$-isomorphism between $C^*$-algebras ${\mathcal{A}}$ and ${\mathcal{A}}_{(H,G)}$.

All the algebras in this paper will be unital associative algebras over the complex field ${\Bbb C}$. The unadorned tensor product $\otimes$ will stand for the usual tensor product over ${\Bbb C}$.
 For general results on Hopf algebras we refer to the books of Abe \cite{E.Abe} and Sweedler\cite{M.E.Sw}. We shall adopt their notations, such as $S$, $\bigtriangleup$, $\varepsilon$ for the antipode, the comultiplication and the counit, respectively. Also we shall use Sweedler-type notation
\begin{eqnarray*}
    \begin{array}{c}
    \bigtriangleup(a)=\sum\limits_{(a)}a_{(1)}\otimes a_{(2)}.
     \end{array}
\end{eqnarray*}

\section{The structure of the observable algebra with respect to $H$}
In this section, suppose that $G$ is a finite group with a normal subgroup $H$,
we will give the algebraic generators for the observable algebra with respect to a normal group $H$. First, we recall some definitions in $G$-spin models with respect to $H$ which will be needed in the sequel \cite{Qiao}.

\begin{definition}$^{\cite{Qiao}}$
$D(H;G)$ is the crossed product of $C(H)$ and group algebra ${\Bbb C}G$, where $C(H)$ denotes the set of complex functions on $H$, with respect to the adjoint action of the latter on the former.
\end{definition}

Using the linear basis elements $(h,g)$ of $D(H;G)$, the structure maps are given by
\begin{eqnarray*}
    \begin{array}{rcll}
(h_1,g_1)(h_2,g_2)&=&\delta_{h_1g_1,g_1h_2}(h_1,g_1g_2),& (\mathrm{multiplication})\\[5pt]
\bigtriangleup(h,g)&=&\sum \limits_{t\in H}(t,g)\otimes (t^{-1}h,g),& (\mathrm{coproduct})\\[5pt]
\varepsilon(h,g)&=&\delta_{h,e},& (\mathrm{counit})\\[5pt]
S(h,g)&=&(g^{-1}h^{-1}g,g^{-1}),& (\mathrm{antipode})\\[5pt]
 (h,g)^*&=&(g^{-1}hg,g^{-1}),& (\mbox{*-operation}),
     \end{array}
\end{eqnarray*}
where $\delta_{g,h}=\left\{\begin{array}{cc}
                          1, & {\mathrm{if}} \ \  g=h \\
                          0, & {\mathrm{if}} \ \ g\neq h
                        \end{array}\right.$.
In \cite{Qiao}, we have shown
$D(H;G)$ is a Hopf $C^*$-algebra, with a unique element $z_{_{(H,G)}}=\frac{1}{|G|}\sum\limits_{g\in G}(e,g)$, called an integral, satisfying for any $a\in D(H;G)$,
$$az_{_{(H,G)}}=z_{_{(H,G)}}a=\varepsilon(a)z_{_{(H,G)}}.$$

As in the traditional case, one can define the local quantum field algebra as follows.

\begin{definition}$^{\cite{Qiao}}$
 The local field ${\mathcal{F}}_{H,\mathrm{loc}}$ determined by $H$ is an associative algebra with a unit $I$ generated by
 $\{\delta_g(x), \rho_h(l)\colon g\in G, h\in H; x\in {\Bbb Z}, l\in {\Bbb Z}+\frac{1}{2}\}$ subject to
\begin{eqnarray*}
    \begin{array}{rcl}
\sum\limits_{g\in G}\delta_g(x)&=&I=\rho_e(l),\\
\delta_{g_1}(x)\delta_{g_2}(x)&=&\delta_{g_1,g_2}\delta_{g_1}(x),\\
 \rho_{h_1}(l)\rho_{h_2}(l)&=&\rho_{h_1h_2}(l),\\
\delta_{g_1}(x)\delta_{g_2}(x')&=&\delta_{g_2}(x')\delta_{g_1}(x),\\
 \rho_h(l)\delta_g(x)&=& \left\{\begin{array}{cc}
                          \delta_{hg}(x)\rho_h(l), & l<x, \\
                          \delta_g(x)\rho_h(l), & l>x,
                        \end{array}\right.\\
 \rho_{h_1}(l)\rho_{h_2}(l')&=& \left\{\begin{array}{cc}
                          \rho_{h_2}(l')\rho_{{h_2}^{-1}h_1{h_2}}(l), & l>l', \\
                          \rho_{{h_1}{h_2}{h_1}^{-1}}(l')\rho_{h_1}(l), & l<l'
                        \end{array}\right.

    \end{array}
\end{eqnarray*}
for $x, x' \in {\Bbb Z};\ l, l'\in {\Bbb Z}+\frac{1}{2} \ {\mathrm{and}} \ h_1, h_2\in H, g_1,g_2\in G$.
\end{definition}

The *-operation is defined on the generators as $\delta_g^\ast(x)=\delta_g(x), \ \rho_h^\ast(l)=\rho_{h^{-1}}(l)$ and is extended to a involution on ${\mathcal{F}}_{H,\mathrm{loc}}$. In this way, ${\mathcal{F}}_{H,\mathrm{loc}}$ becomes a unital *-algebra. Using the $C^*$-inductive limit, ${\mathcal{F}}_{H,\mathrm{loc}}$ can be extended to a $C^*$-algebra ${\mathcal{F}}_H$, called the field algebra of $G$-spin models determined by a normal subgroup $H$.
There is an action $\gamma$ of $D(H;G)$ on ${\mathcal{F}}_H$ in the following. For $x\in {\Bbb Z}; \ l\in {\Bbb Z}+\frac{1}{2} \ {\mathrm{and}} \ h\in H, g\in G$, set
\begin{eqnarray*}
    \begin{array}{rcl}
(h,g)\delta_f(x)&=&\delta_{h,e}\delta_{gf}(x), \ \ \forall f\in G,\\
(h,g)\rho_t(l)&=&\delta_{h,gtg^{-1}}\rho_h(l), \ \ \forall t\in H.
 \end{array}
\end{eqnarray*}
One can check that ${\mathcal{F}}_H$ is a $D(H;G)$-module algebra \cite{Qiao}.

Set $${\mathcal{A}}_{(H,G)}=\{F\in {\mathcal{F}}_H\colon a(F)=\varepsilon(a)(F), \ \forall a\in D(H;G)\}.$$ We call it an observable algebra related to $H$ in the field algebra ${\mathcal{F}}$ of $G$-spin models. Furthermore, one can show that ${\mathcal{A}}_{(H,G)}$ is a nonzero $C^*$-subalgebra of ${\mathcal{F}}_H$, and $${\mathcal{A}}_{(H,G)}=\{F\in {\mathcal{F}}_H\colon z_{_{(H,G)}}(F)=F\}\equiv z_{_{(H,G)}}({\mathcal{F}}_H).$$

Now, we will discuss the concrete construction of ${\mathcal{A}}_{(H,G)}$. In order to do this, for $g\in G$, $x\in\Bbb{Z}$, and $l\in\Bbb{Z}+\frac{1}{2}$, put
\begin{eqnarray*}
    \begin{array}{rcl}
v_g(x)&=&\sum\limits_{h\in G}\varrho_{hg^{-1}h^{-1}}(x-\frac{1}{2})\delta_h(x)
\varrho_{hgh^{-1}}(x+\frac{1}{2}),\\[5pt]
w_g(l)&=&\sum\limits_{h\in G}\delta_h(l-\frac{1}{2})\delta_{hg}(l+\frac{1}{2}).
\end{array}
\end{eqnarray*}

\begin{theorem}
The observable algebra ${\mathcal{A}}_{(H,G)}$ related to $H$ is a unital $C^*$-subalgebra of ${\mathcal{F}}_H$ generated by
$$\left\{v_h(x), w_g(l)\colon h\in H, g\in G, x\in\Bbb{Z}, l\in\Bbb{Z}+\frac{1}{2}\right\}.$$
\end{theorem}

\begin{proof}
Since $z_{_{(G,G)}}=z_{_{(G,G)}}z_{_{(G,\{e\})}}$, where $z_{_{(G,\{e\})}}$ is the unique integral of Hopf algebra $D(G, \{e\})$, we have that
\begin{eqnarray*}
\gamma_{z_{_{(G,G)}}}({\mathcal{F}}_H)=\gamma_{z_{_{(G,G)}}z_{_{(G,\{e\})}}}({\mathcal{F}}_H)
=\gamma_{z_{_{(G,G)}}}(\gamma_{z_{_{(G,\{e\})}}}({\mathcal{F}}_H)).
\end{eqnarray*}
$\gamma_{z_{_{(G,\{e\})}}}$ is the projection to operators with trivial twist (\cite{K.Szl}). Hence,
$\gamma_{z_{_{(G,\{e\})}}}({\mathcal{F}}_H)$ is generated by
\begin{eqnarray*}
    \left\{\delta_g(x), v_h(x)\colon
 g\in G, h\in H, x\in\Bbb{Z}\right\}.
\end{eqnarray*}
 Moreover, we can obtain that
$\gamma_{z_{_{(G,G)}}}({\mathcal{F}}_H)$ is generated by
\begin{eqnarray*}
    \left\{w_g(l), v_h(x)\colon
 g\in G, h\in H, x\in\Bbb{Z}, l\in\Bbb{Z}+\frac{1}{2}\right\},
\end{eqnarray*}
 since
\begin{eqnarray*}
    \gamma_{z_{_{(G,G)}}}\left(\delta_{g_1}(1)\delta_{g_2}(2)\cdots\delta_{g_n}(n)\right)
=\frac{1}{|G|}w_{{g_1}^{-1}{g_2}}(\frac{3}{2})w_{{g_2}^{-1}{g_3}}(\frac{5}{2})
\cdots w_{{g_{n-1}}^{-1}{g_n}}(n-\frac{1}{2}).
\end{eqnarray*}

Now let us consider $\gamma_{z_{_{(G,G)}}}|_{{\mathcal{F}}_H}$, the restriction of $\gamma_{z_{(G,G)}}$ on ${\mathcal{F}}_H$, and $\gamma_{z_{_{(H,G)}}}$ as  projections on ${\mathcal{F}}_{H}$, we have that
$$\gamma_{z_{_{(G,G)}}}|_{{\mathcal{F}}_H}\gamma_{z_{_{(H,G)}}}
=\gamma_{z_{_{(H,G)}}}\gamma_{z_{_{(G,G)}}}|_{{\mathcal{F}}_H}=\gamma_{z_{_{(H,G)}}},$$
then $$\gamma_{z_{_{(H,G)}}}\leq \gamma_{z_{_{(G,G)}}}|_{{\mathcal{F}}_H},$$
which implies $\gamma_{z_{_{(H,G)}}}({\mathcal{F}}_H)\subseteq \gamma_{z_{_{(G,G)}}}({\mathcal{F}}_H)$.
Again, for $g\in G, h\in H, x\in\Bbb{Z}, l\in\Bbb{Z}+\frac{1}{2},$
$$\gamma_{z_{_{(H,G)}}}(w_g(l))=w_g(l), \ \ \ \  \gamma_{z_{_{(H,G)}}}(v_h(x))=v_h(x).$$
Hence, $\gamma_{z_{_{(H,G)}}}({\mathcal{F}}_H)$ is generated by
 \begin{eqnarray*}
\left\{v_h(x), w_g(l)\colon
 h\in H, g\in G, x\in\Bbb{Z}, l\in\Bbb{Z}+\frac{1}{2}\right\}.
 \end{eqnarray*}
\end{proof}

\section{The characterization of the observable algebra ${\mathcal{A}}_{(H,G)}$}
In this section, we identify $H$ with the group algebra $\Bbb{C}H$, and $\widehat{G}$ the set of complex functions on $G$.
We will construct iterated twisted tensor product algebras of any number of factors.
To do this, let us recall briefly the concept of a twisted tensor
product of algebras (\cite{A.Cap}).

\begin{definition}
Let $A$ and $B$ be two unital associative algebras.
Suppose that $R\colon B\bigotimes A\rightarrow A\bigotimes B$ is a linear map such that
\begin{eqnarray*}
    \begin{array}{rcl}
R\circ(id_B\bigotimes m_A)&=&(m_A\bigotimes id_B)\circ(id_A\bigotimes R)\circ(R\bigotimes id_A),\\
R\circ(m_B\bigotimes id_A)&=&(id_A\bigotimes m_B)\circ(R \bigotimes id_B)\circ(id_B\bigotimes R),
 \end{array}
\end{eqnarray*}
then $m_R=(m_A\bigotimes m_B)\circ(id_A\bigotimes R\bigotimes id_B)$ is an associative product on $A\bigotimes B$. Here, $m_A$ and $m_B$ denote the multiplication in algebras $A$ and $B$, respectively.
In this case, we call $R$ a twisting map, and $(A\bigotimes B, m_R)$
a twisted tensor product of $A$ and $B$, which has $A \bigotimes B$ as underlying vector space endowed with the multiplication $m_R$, simply denoted by
$A \bigotimes_R B$.
 Using a Sweedler-type notation, we denote by
$R(b\otimes a)=a_R\otimes b_R$ for $a\in A$ and $b\in B$.
\end{definition}

Observe that the multiplication $m_R$ in the twisted  product $A \bigotimes_R B$ of algebras $A$ and $B$ can be given in the following form:
\begin{eqnarray*}
    \begin{array}{c}
(a\otimes b)(a'\otimes b')=aa'_R\otimes b_Rb',
 \end{array}
\end{eqnarray*}
where, as already mentioned, the Sweedler-type notation for the twisting map $R$ has been used, i.e.,
$R(b\otimes a)=a_R\otimes b_R$ for $a\in A$ and $b\in B$.

\begin{example}
(1) Consider the usual flip $\tau\colon B\bigotimes A\rightarrow A\bigotimes B$ defined by
$$\tau(b\otimes a)=a\otimes b.$$
It is obvious that $\tau$ satisfies all conditions for the twisting map, and then it give rise to the standard tensor product of algebras $A\bigotimes B$.

(2) Suppose that $M$ is a Hopf algebra over $\Bbb{C}$, and $B$ is a (left) $M$-module algebra, that is, $B$ is an algebra which is a left $M$-module such that $m\cdot(ab)=\sum\limits_{(m)}(m_{(1)}\cdot a)(m_{(2)}\cdot b)$ and
$m\cdot 1_B=\varepsilon(m)1_B$, for all $a,b\in B$, $m\in M$.

Let the map $R\colon M\bigotimes B \rightarrow B\bigotimes M$ be defined by
\begin{eqnarray*}
    \begin{array}{c}
R(m\otimes b)=\sum\limits_{(m)}(m_{(1)}\cdot b)\otimes m_{(2)}.
 \end{array}
\end{eqnarray*}
One can show that $R$ is a twisting map, and then obtain the algebra $B\bigotimes_R M$, which is $B\bigotimes M$ as a vector space with multiplication
\begin{eqnarray*}
    \begin{array}{c}
    (a\otimes m)(b\otimes n)=\sum\limits_{(m)}a(m_{(1)}\cdot b)\otimes m_{(2)}n.
 \end{array}
\end{eqnarray*}
From the definition of the smash product, it is easy to see that $B\bigotimes_R M$ coincides with the ordinary smash product $B\# M$ introduced in \cite{R.G.He}.
\end{example}

In order to study the construction of iterated twisted tensor products,
we consider three twisted tensor products $A\bigotimes_{R_1} B$, $B\bigotimes_{R_2} C$ and $A\bigotimes_{R_3} C$, and the maps
\begin{eqnarray*}
    \begin{array}{c}
T_1 \colon C\bigotimes (A\bigotimes_{R_1} B) \rightarrow (A\bigotimes_{R_1} B)\bigotimes C
 \end{array}
\end{eqnarray*}
 defined by $T_1=(id_A\bigotimes R_2)\circ(R_3\bigotimes id_B)$
 and
 \begin{eqnarray*}
    \begin{array}{c}
 T_2 \colon (B\bigotimes_{R_2} C)\bigotimes A \rightarrow A\bigotimes (B\bigotimes_{R_2} C) \end{array}
\end{eqnarray*}
 defined by $T_2=(R_1\bigotimes id_C)\circ(id_B\bigotimes R_3)$
 associated to $R_1, R_2$ and $R_3$. The following lemma states a sufficient and necessary condition
 ensuring that both $T_1$ and $T_2$ are twisting maps.

\begin{lemma}$^{\cite{P.Jar}}$
The following statements are equivalent:

 (1)\ $T_1$ is a twisting map.

 (2)\ $T_2$ is a twisting map.

 (3)\ The maps $R_1, R_2, R_3$ satisfy the following compatibility condition (called the hexagon equation):
\begin{eqnarray*}
    \begin{array}{c}
 (id_A\bigotimes R_2)\circ (R_3 \bigotimes id_B)\circ(id_C\bigotimes R_1)=(R_1\bigotimes id_C)\circ (id_B\bigotimes R_3)\circ (R_2\bigotimes id_A).
  \end{array}
\end{eqnarray*}
Moreover, if all the three conditions are satisfied, then $A\bigotimes_{T_2} (B\bigotimes_{R_2} C)$ and $(A\bigotimes_{R_1} B)\bigotimes_{T_1} C$ are equal. In this case, we will denote this algebra by $A\bigotimes_{R_1} B\bigotimes_{R_2} C$, which is called the iterated twisted tensor product.
\end{lemma}

Now, we consider the group algebra $\Bbb{C}G$, endowed with a comultiplication $\bigtriangleup(g)=g\otimes g$, a counit $\varepsilon(g)=1$, antipode $S(g)=g^{-1}$ and $g^*=g^{-1}$ for all $g\in G$, such that $\Bbb{C}G$ is a $C^*$-Hopf algebra. From now on, we use $G$ for $\Bbb{C}G$. The dual of $G$ is $\widehat{G}$, with $\bigtriangleup(\delta_g)=\sum\limits_{t\in G}\delta_t\otimes\delta_{t^{-1}g}$, $\varepsilon(\delta_g)=\delta_{g,e}$, $S(\delta_g)=\delta_{g^{-1}}$ and $\delta_g^*=\delta_g$ for all $g\in G$.
There is a natural pairing between $G$ and $\widehat{G}$ given by
\begin{eqnarray*}
    \begin{array}{c}
\langle \ , \ \rangle\colon G\bigotimes \widehat{G}\rightarrow \Bbb{C}, \ \
g\otimes \delta_s \mapsto\langle g, \delta_s \rangle\colon=\delta_s(g),\\
\langle \ , \ \rangle\colon \widehat{G} \bigotimes G\rightarrow \Bbb{C}, \ \
\delta_s\otimes g \mapsto\langle\delta_s, g \rangle\colon=\delta_s(g).
  \end{array}
\end{eqnarray*}
Associated to this pairing we have the natural action of $G$ on $\widehat{G}$ and that on $\widehat{G}$ on $G$ given by the Sweedler's arrows:
\begin{eqnarray*}
    \begin{array}{c}
g\rightarrow \delta_s \colon= \sum\limits_{(\delta_s)}{\delta_s}_{(1)}\langle g,{\delta_s}_{(2)}\rangle=\delta_{sg^{-1}},\\
\delta_s\rightarrow g \colon= \sum\limits_{(g)}g_{(1)}\langle \delta_s,g_{(2)}\rangle=g\delta_{s}(g),
  \end{array}
\end{eqnarray*}
where $\bigtriangleup(\delta_s)=\sum\limits_{(\delta_s)}{\delta_s}_{(1)}\otimes {\delta_s}_{(2)}=\sum\limits_{t\in G}{\delta_t}\otimes\delta_{t^{-1}s}$, and $\bigtriangleup(g)=\sum\limits_{(g)}g_{(1)}\otimes g_{(2)}=g\otimes g$.

For every $n\in \Bbb{Z}$, take $A_n\colon=H$ if $n$ is even and $A_n\colon=\widehat{G}$ if $n$ is odd, and define the maps:
$$\begin{array}{ccrcl}
 R_{2n,2n+1}&\colon &A_{2n+1}\bigotimes A_{2n}&\longrightarrow& A_{2n}\bigotimes A_{2n+1}\\
   &&\delta_g\otimes h&\longmapsto& \sum\limits_{(\delta_g)}({\delta_g}_{(1)} \rightarrow h ) \otimes {\delta_g}_{(2)}=h\otimes \delta_{h^{-1}g}, \\
 R_{2n-1,2n}&\colon &A_{2n}\bigotimes A_{2n-1}&\longrightarrow& A_{2n-1}\bigotimes A_{2n}\\
   && h\otimes\delta_g&\longmapsto& \sum\limits_{(h)}(h_{(1)}\rightarrow \delta_g)\otimes h_{(2)}=\delta_{g h^{-1}}\otimes h,\\
 R_{i,j}&\colon &A_j\bigotimes A_i&\longrightarrow& A_i\bigotimes A_j\\
   && x_j\otimes x_i&\longmapsto& x_i\otimes x_j,  \ \ \ {\mathrm{if}} \ j-i\geq 2.
\end{array}$$
It is clear that all of them are twisting maps.

\begin{proposition}
$R_{0,1}, R_{1,2}, R_{0,2}$ and $R_{1,2}, R_{2,3}, R_{1,3}$ are compatible, respectively.
\end{proposition}

\begin{proof}
It suffices to show that $R_{0,1}, R_{1,2}, R_{0,2}$ are compatible. To do this,
 apply the left-hand side of the hexagon equation to a generator $h_2\otimes \delta_g\otimes h_1$ of $A_2\bigotimes A_1\bigotimes A_0$, we get
\begin{eqnarray*}
    \begin{array}{rcl}
&&(id_H\otimes R_{1,2})\circ(R_{0,2}\otimes id_{\widehat{G}})\circ(id_H\otimes R_{0,1})(h_2\otimes \delta_g\otimes h_1)\\
&=&(id_H\otimes R_{1,2})\circ(R_{0,2}\otimes id_{\widehat{G}})(h_2\otimes h_1\otimes\delta_{h_1^{-1}g})\\
&=&(id_H\otimes R_{1,2})(h_1\otimes h_2\otimes\delta_{h_1^{-1}g})\\
&=&h_1\otimes \delta_{h_1^{-1}gh_2^{-1}}\otimes h_2.
\end{array}
\end{eqnarray*}
On the other hand, for the right hand side we obtain that
\begin{eqnarray*}
    \begin{array}{rcl}
&&(R_{0,1}\otimes id_H)\circ(id_{\widehat{G}}\otimes R_{0,2})\circ(R_{1,2}\otimes id_H)(h_2\otimes \delta_g\otimes h_1)\\
&=&(R_{0,1}\otimes id_H)\circ(id_{\widehat{G}}\otimes R_{0,2})(\delta_{g h_2^{-1}}\otimes h_2 \otimes h_1)\\
&=&(R_{0,1}\otimes id_H)(\delta_{g h_2^{-1}}\otimes h_1\otimes h_2)\\
&=&h_1\otimes \delta_{h_1^{-1}gh_2^{-1}}\otimes h_2.
\end{array}
\end{eqnarray*}
Now, we have shown $R_{0,1}, R_{1,2}$ and $R_{0,2}$ are compatible. Similarly, one can prove $R_{1,2}, R_{2,3}$ and $R_{1,3}$ are compatible.
\end{proof}

It follows from Proposition 3.1 and Lemma 3.1 that one can construct the algebras $A_0\bigotimes_{R_{0,1}}A_1\bigotimes_{R_{1,2}} A_2$ and $A_1\bigotimes_{R_{1,2}} A_2\bigotimes_{R_{2,3}}A_3$, in which the multiplications can be given, respectively, by the formulas
$$(h_1\otimes \delta_{g_1}\otimes f_1)(h_2\otimes \delta_{g_2}\otimes f_2)=h_1h_2\otimes\delta_{h_2^{-1}g_1}\delta_{g_2f_1^{-1}}\otimes f_1f_2,$$
$$(\delta_{g_1}\otimes h_1\otimes \delta_{s_1})(\delta_{g_2}\otimes h_2\otimes \delta_{s_2})=\delta_{g_1}\delta_{g_2h_1^{-1}}\otimes h_1h_2\otimes\delta_{h_2^{-1}s_1}\delta_{s_2}.$$
Moreover, we can define a map
$$\begin{array}{cccc}
\theta \colon& A_0\bigotimes_{R_{0,1}}A_1\bigotimes_{R_{1,2}} A_2 &\longrightarrow& A_0\bigotimes_{R_{0,1}}A_1\bigotimes_{R_{1,2}} A_2\\
&h_1\otimes \delta_g\otimes h_2 &\longrightarrow& h_1^{-1}\otimes \delta_{h_1gh_2}\otimes h_2^{-1}.
\end{array}$$
Then $\theta$ satisfies the following properties: for any $x,y\in A_0\bigotimes_{R_{0,1}}A_1\bigotimes_{R_{1,2}} A_2$,
\begin{eqnarray*}
\begin{array}{c}
\theta(\theta(x))=x, \ \ \ \ \ \theta(xy)=\theta(y)\theta(x).
\end{array}
\end{eqnarray*}
Thus $A_0\bigotimes_{R_{0,1}}A_1\bigotimes_{R_{1,2}} A_2$ is a *-algebra by means of \begin{eqnarray*}
\begin{array}{c}
(h_1\otimes \delta_g\otimes h_2)^*=\theta(h_1\otimes \delta_g\otimes h_2)
 \end{array}
\end{eqnarray*}
for $h_1\otimes \delta_g\otimes h_2\in A_0\bigotimes_{R_{0,1}}A_1\bigotimes_{R_{1,2}} A_2$.

Similarly, for $\delta_g\otimes h\otimes \delta_s\in A_1\bigotimes_{R_{1,2}} A_2\bigotimes_{R_{2,3}}A_3$, set
$$(\delta_g\otimes h\otimes \delta_s)^*=\delta_{gh}\otimes h^{-1}\otimes \delta_{hs}.$$
Then $A_1\bigotimes_{R_{1,2}} A_2\bigotimes_{R_{2,3}}A_3$ is also a *-algebra.

Moreover, we have the following proposition.
\begin{proposition}
 $A_0\bigotimes_{R_{0,1}}A_1\bigotimes_{R_{1,2}} A_2$ and $A_1\bigotimes_{R_{1,2}} A_2\bigotimes_{R_{2,3}}A_3$ are $C^*$-algebras.
\end{proposition}

\begin{proof}
Let ${\mathcal{H}}=L^2(\widehat{G},h)$ be a Hilbert space with inner product
\begin{eqnarray*}
\begin{array}{c}
\langle\varphi,\phi\rangle=\frac{1}{|G|}\sum\limits_{g\in G}\varphi(g)\overline{\phi(g)},
\end{array}
\end{eqnarray*}
and ${\mathcal{H}}_{0,2}\triangleq{\mathcal{H}}\otimes{\mathcal{H}}$.

Consider the map $\pi_{0,2}\colon A_0\bigotimes_{R_{0,1}}A_1\bigotimes_{R_{1,2}} A_2\rightarrow {\mathrm{End}}{\mathcal{H}}_{0,2}$ be given by
\begin{eqnarray*}
\begin{array}{rcl}
(\pi_{0,2}(g^{(0)})\psi)(g_0,g_2)&=&\psi(g_0g^{(0)},g_2)\\
(\pi_{0,2}(g^{(2)})\psi)(g_0,g_2)&=&\psi(g_0,g_2g^{(2)})\\
(\pi_{0,2}(\delta_g)\psi)(g_0,g_2)&=&\delta_g(g_0^{-1}g_2)\psi(g_0,g_2),
\end{array}
\end{eqnarray*}
where $\delta_g\in A_1$, $g^{(0)}\in A_0$, $g^{(2)}\in A_2$ and $g_0,g_2\in G$. One can show that $(\pi_{0,2},{\mathcal{H}}_{0,2})$ is a faithful *-representation of $A_{0,2}$.

 For $h_1\otimes \delta_g\otimes h_2\in A_0\bigotimes_{R_{0,1}}A_1\bigotimes_{R_{1,2}} A_2$, set \begin{eqnarray*}
\begin{array}{c}
\|h_1\otimes \delta_g\otimes h_2\|=\|\pi_{0,2}(h_1\otimes \delta_g\otimes h_2)\|,
\end{array}
\end{eqnarray*}
then $(A_0\bigotimes_{R_{0,1}}A_1\bigotimes_{R_{1,2}} A_2, \|\cdot\|)$ is a $C^*$-algebra of finite dimension.

As to $A_{1,3}$, we define faithful *-representation $(\pi_{1,3},{\mathcal{H}}_{1,3})$ of $A_{1,3}$, where ${\mathcal{H}}_{1,3}=\widetilde{\mathcal{H}}\otimes\widetilde{\mathcal{H}}$ with $\widetilde{\mathcal{H}}=L^2(G,\delta_e)$, $\delta_e\in \widehat{G}$ being the Haar measure on the group algebra $G$. Hence, $A_1\bigotimes_{R_{1,2}} A_2\bigotimes_{R_{2,3}}A_3$ is a $C^*$-algebra.
\end{proof}

In the following, by induction, we will construct an iterated twisted tensor product of any number of factors.

\begin{proposition}
The three maps $R_{i,j}, R_{j,k}$ and $R_{i,k}$ are compatible for any $i<j<k$.
\end{proposition}

\begin{proof}
Let us distinguish among several cases:

If $j-i\geq 2$ and $k-j\geq 2$, all three maps are just usual flips, and thus $R_{i,j}, R_{j,k}$ and $R_{i,k}$ are compatible.

If $j-i=1$ and $k-j\geq 2$, then we have that both $R_{i,k}$ and $R_{j,k}$ are usual flips. Hence, $R_{i,j}, R_{j,k}$ and $R_{i,k}$ are compatible. Indeed, for any $x\otimes y\otimes z\in A_k\bigotimes A_j\bigotimes A_i$, we have
\begin{eqnarray*}
    \begin{array}{rcl}
&&(id_{A_i}\otimes R_{j,k})\circ(R_{i,k}\otimes id_{A_j})\circ(id_{A_k}\otimes R_{i,j})(x\otimes y\otimes z)\\
&=&(id_{A_i}\otimes R_{j,k})\circ(R_{i,k}\otimes id_{A_j})(x\otimes z_{_{R_{i,j}}}\otimes y_{_{R_{i,j}}})\\
&=&(id_{A_i}\otimes R_{j,k})(z_{_{R_{i,j}}}\otimes x \otimes y_{_{R_{i,j}}})\\
&=&z_{_{R_{i,j}}}\otimes y_{_{R_{i,j}}} \otimes x,
\end{array}
\end{eqnarray*}
and
\begin{eqnarray*}
    \begin{array}{rcl}
&&(R_{i,j}\otimes id_{A_k})\circ(id_{A_j}\otimes R_{i,k})\circ(R_{j,k}\otimes id_{A_i})(x\otimes y\otimes z)\\
&=&(R_{i,j}\otimes id_{A_k})\circ(id_{A_j}\otimes R_{i,k})(y\otimes x\otimes z)\\
&=&(R_{i,j}\otimes id_{A_k})(y\otimes z\otimes x)\\
&=&z_{_{R_{i,j}}}\otimes y_{_{R_{i,j}}} \otimes x.
\end{array}
\end{eqnarray*}
Moreover, for any three twisting maps, if two of them are usual flips, then these three twisting maps are compatible. This statement implies $R_{i,j}, R_{j,k}$ and $R_{i,k}$ are compatible for $j-i\geq 2$ and $k-j=1$.

If $j-i=1$ and $k-j=1$, then $R_{i,j}, R_{j,k}$ and $R_{i,k}$ are compatible, the proof of which is the same as that of $R_{0,1}, R_{1,2}, R_{0,2}$ and $R_{1,2}, R_{2,3}, R_{1,3}$ are compatible.
\end{proof}

Assume that we have $n$ algebras $A_1, A_2, \cdots, A_n$ with a twisting map $R_{i,j}\colon A_j\bigotimes A_i\rightarrow A_i\bigotimes A_j$ for any $i<j$, and such that for every $i<j<k$ the maps $R_{i,j}, R_{j,k}$ and $R_{i,k}$ are compatible. Define now for any $i<n-1$ the map
\begin{eqnarray*}
    \begin{array}{c}
T_{n-1,n}^i\colon (A_{n-1}\bigotimes_{R_{n-1,n}} A_n)\bigotimes A_i\rightarrow A_i\bigotimes(A_{n-1}\bigotimes_{R_{n-1,n}} A_n)
\end{array}
\end{eqnarray*}
given by $T_{n-1,n}^i=( R_{i,n-1}\otimes A_{n})\circ(A_{n-1}\otimes R_{i,n})$, which are twisting maps for each $i\in\Bbb{Z}$, as the maps $R_{i,n-1}, R_{i,n}$ and $R_{n-1,n}$ are compatible (Lemma 3.1).
Moreover, for every $i<j<n-1$, the compatibility for $R_{i,j},R_{i,n-1},R_{j,n-1}$ and $R_{i,j},R_{i,n},R_{j,n}$ implies that for $R_{i,j}, T_{n-1,n}^i$ and $T_{n-1,n}^j$.
Hence, we can apply the induction hypothesis to the $n-1$ algebras $A_1, A_2, \cdots, A_{n-2}$ and $A_{n-1}\bigotimes_{R_{n-1,n}} A_n$, and can construct the twisted product of these $n-1$ factors, then we obtain the $C^*$-algebra
\begin{eqnarray*}
    \begin{array}{c}
A_1\bigotimes_{R_{1,2}}\cdots\bigotimes_{R_{n-3,n-2}}A_{n-2}
\bigotimes_{T_{n-1,n}^{n-2}}(A_{n-1}\bigotimes_{R_{n-1,n}} A_n).
\end{array}
\end{eqnarray*}
Notice that the maps $R_{n-2,n-1}, R_{n-1,n}$ and $R_{n-2,n-1}$ are compatible, which implies that \begin{eqnarray*}
    \begin{array}{c}
A_{n-2}\bigotimes_{T_{n-1,n}^{n-2}}(A_{n-1}\bigotimes_{R_{n-1,n}} A_n)=
(A_{n-2}\bigotimes_{R_{n-2,n-1}}A_{n-1})\bigotimes_{T_{n-2,n-1}^n} A_n.
\end{array}
\end{eqnarray*}
Now, we get the $C^*$-algebra
\begin{eqnarray*}
    \begin{array}{c}
A_1\bigotimes_{R_{1,2}}\cdots\bigotimes_{R_{n-3,n-2}}A_{n-2}
\bigotimes_{R_{n-2,n-1}}A_{n-1}\bigotimes_{R_{n-1,n}} A_n.
\end{array}
\end{eqnarray*}
In particular, for any $n,m\in \Bbb{Z}$ with $n<m$, we can define the $C^*$-algebras of finite dimension
\begin{eqnarray*}
    \begin{array}{c}
A_{n,m}\colon=A_n\bigotimes_{R_{n,n+1}}A_{n+1}\bigotimes_{R_{n+1,n+2}}\cdots A_{m-1}
\bigotimes_{R_{m-1,m}}A_m.
\end{array}
\end{eqnarray*}
If $n<n'$ and $m'<m$, then one can check that $A_{n',m'}\subseteq A_{n,m}$, with the map $i\colon A_{n',m'}\rightarrow A_{n,m}$ defined by
\begin{eqnarray*}
    \begin{array}{c}
i(x_{n'}\otimes x_{n'+1}\otimes\cdots\otimes x_{m'})
=1_{A_n}\otimes\cdots\otimes 1_{A_{n'-1}}\otimes x_{n'}\otimes x_{n'+1}\otimes\cdots\otimes x_{m'}\otimes 1_{A_{m'+1}}\otimes\cdots\otimes 1_{{A_m}}
\end{array}
\end{eqnarray*}
 is a $C^*$-algebra homomorphism preserving the norm, where we use $1_{A_i}$ for the identity in $A_i$ for $i\in\Bbb{Z}$.

 We write ${\mathcal{A}}$ for the $C^*$-inductive limit of $A_{n,m}$ with $n,m\in\Bbb{Z}$:
\begin{eqnarray*}
    \begin{array}{c}
{\mathcal{A}}\colon=\overline{\bigcup\limits_{n<m}A_{n,m}}.
\end{array}
\end{eqnarray*}
That is
\begin{eqnarray*}
    \begin{array}{c}
{\mathcal{A}}\colon=\cdots H\bigotimes_{R_{-2,-1}} {\widehat{G}}\bigotimes_{R_{-1,0}} H\bigotimes_{R_{0,1}}{\widehat{G}}\bigotimes_{R_{1,2}} H\bigotimes_{R_{2,3}}{\widehat{G}}\cdots,
\end{array}
\end{eqnarray*}
where the dots include a $C^*$-inductive limit procedure.

The following theorem is the main result of this paper, which gives a characterization of ${\mathcal{A}}_{(H,G)}$, the observable algebra related to a normal subgroup $H$ of $G$.

\begin{theorem}
The iterated crossed product ${\mathcal{A}}$ is $C^*$-isomorphic to ${\mathcal{A}}_{(H,G)}$.
\end{theorem}

\begin{proof}
For a finite interval $\Lambda$, let
\begin{eqnarray*}
    \begin{array}{c}
    {\mathcal{A}_H}(\Lambda)=\left\langle v_h(x), w_g(l): h\in H, g\in G, x\in \Lambda\cap\Bbb{Z}, l\in \Lambda\cap(\Bbb{Z}+\frac{1}{2})\right\rangle.
     \end{array}
\end{eqnarray*}
It follows from Theorem 2.1 that
\begin{eqnarray*}
    \begin{array}{c}
    {\mathcal{A}}_{(H,G)}=\overline{\bigcup\limits_\Lambda{\mathcal{A}_H}(\Lambda)},
 \end{array}
\end{eqnarray*}
where the union is taken over finite intervals $\Lambda$ and the bar denotes uniform closure.

Consider the map
\begin{eqnarray*}
    \begin{array}{c}
    \Phi_{0,1}\colon A_{0,2}=H\bigotimes_{R_{0,1}}{\widehat{G}}\bigotimes_{R_{1,2}} H \longrightarrow {\mathcal{A}_H}(\Lambda_{0,1})=\langle v_h(0),w_g(\frac{1}{2}), v_f(1)\colon h,f\in H, g\in G \rangle
      \end{array}
\end{eqnarray*}
defined on a generator $h\otimes\delta_g\otimes f$ of $H\bigotimes_{R_{0,1}}{\widehat{G}}\bigotimes_{R_{1,2}} H$ by
\begin{eqnarray*}
    \begin{array}{c}
    \Phi_{0,1}(h\otimes\delta_g\otimes f)=v_h(0)w_g(\frac{1}{2})v_f(1)
      \end{array}
\end{eqnarray*}
Observe that
\begin{eqnarray*}
    \begin{array}{rcl}
v_{h_1}(0)w_{g_1}(\frac{1}{2})v_{f_1}(1)v_{h_2}(0)w_{g_2}(\frac{1}{2})v_{f_2}(1)
&=&v_{h_1}(0)w_{g_1}(\frac{1}{2})v_{h_2}(0)v_{f_1}(1)w_{g_2}(\frac{1}{2})v_{f_2}(1)\\
&=&v_{h_1}(0)v_{h_2}(0)w_{h_2^{-1}g_1}(\frac{1}{2})v_{f_1}(1)w_{g_2}(\frac{1}{2})v_{f_2}(1)\\
&=&v_{h_1h_2}(0)w_{h_2^{-1}g_1}(\frac{1}{2})w_{g_2f_1^{-1}}(\frac{1}{2})v_{f_1}(1)v_{f_2}(1)\\
&=&v_{h_1h_2}(0)w_{h_2^{-1}g_1}(\frac{1}{2})w_{g_2f_1^{-1}}(\frac{1}{2})v_{f_1f_2}(1),
      \end{array}
\end{eqnarray*}
where we use the commutation relations of the $v,w$ generators
\begin{eqnarray*}
    \begin{array}{rcl}
    v_{h_1}(x)v_{h_1}(x)&=&v_{h_1h_2}(x),\\
      w_{g_1}(l)w_{g_2}(l)&=&\delta_{g_1,g_2}w_{g_1}(l),\\
v_h(x)w_g(x+\frac{1}{2})&=&w_{hg}(x+\frac{1}{2})v_h(x),\\
v_h(x)w_g(x-\frac{1}{2})&=&w_{gh^{-1}}(x-\frac{1}{2})v_h(x),
      \end{array}
\end{eqnarray*}
other pairs of $v$ and$/$or $w$ fields commute.

On the other hand,
\begin{eqnarray*}
    \begin{array}{c}
    (h_1\otimes\delta_{g_1}\otimes f_1)(h_2\otimes\delta_{g_2}\otimes f_2)
=h_1h_2\otimes\delta_{h_2^{-1}g_1}\delta_{g_2f_1^{-1}}\otimes f_1f_2.
      \end{array}
\end{eqnarray*}
Thus, $\Phi_{0,1}$ is an algebra homomorphism. And we have that
\begin{eqnarray*}
    \begin{array}{rcl}
   (\Phi_{0,1}(h\otimes\delta_g\otimes f))^*&=&(v_h(0)w_g(\frac{1}{2})v_f(1))^*\\
   &=&v_{f^{-1}}(1)w_g(\frac{1}{2})v_{h^{-1}}(0)\\
&=&v_{h^{-1}}(0)v_{f^{-1}}(1)w_{hg}(\frac{1}{2})\\
&=&v_{h^{-1}}(0)w_{hgf}(\frac{1}{2})v_{f^{-1}}(1)\\
&=&\Phi_{0,1}(h^{-1}\otimes\delta_{hgf}\otimes f^{-1})\\
&=&\Phi_{0,1}((h\otimes\delta_g\otimes f)^*),
      \end{array}
\end{eqnarray*}
where we use the properties: $w_g(l)$ is a self-adjoint projection in ${\mathcal{A}_H}(\Lambda)$, and $v_h(x)$ is a unitary element in ${\mathcal{A}_H}(\Lambda)$ for any $l\in \Lambda\cap(\Bbb{Z}+\frac{1}{2}), x\in \Lambda\cap\Bbb{Z}$. Hence, $\Phi_{0,1}$ is a $C^*$-homomorphism. By Theorem 2.1.7 in \cite{GJM}, we know $\Phi_{0,1}$ is norm-decreasing.

Also, $\Phi_{0,1}$ is bijective, which together with the open mapping theorem yields that $\Phi_{0,1}$ is a $C^*$-isomorphism
between $C^*$-algebras $A_{0,2}$ and ${\mathcal{A}_H}(\Lambda_{0,1})$.

By induction, we can build a $C^*$-isomorphism $\Phi_{-n,m}$ between $A_{-2n,2m}$ and ${\mathcal{A}_H}(\Lambda_{-n,m})$, for any $n,m\in\Bbb{Z}$.

Now, because of the last relation we can define a $C^*$-isomorphism
\begin{eqnarray*}
    \begin{array}{c}
\Phi\colon \bigcup\limits_{n<m}A_{-2n,2m}\rightarrow
\bigcup\limits_{n<m}{\mathcal{A}_H}(\Lambda_{-n,m})
       \end{array}
\end{eqnarray*}
by $\Phi|_{{\mathcal{A}}_{-2n,2m}}=\Phi_{-n,m}$.
Since each $\Phi_{-n,m}$ is an isometry, then $\Phi$ is an isometry, and $\Phi$
can therefore be extended to the map of $\overline{\bigcup\limits_{n<m}A_{-2n,2m}}$ onto $\overline{\bigcup\limits_{n<m}{\mathcal{A}_H}(\Lambda_{-n,m})}$ by continuity.
Since all the operations in the definition of a $C^*$-algebra are norm continuous,
this extended map is an isomorphism.
Hence,
$\overline{\bigcup\limits_{n<m}A_{-2n,2m}}$ is $C^*$-isomorphic to $\overline{\bigcup\limits_{n<m}{\mathcal{A}_H}(\Lambda_{-n,m})}$.

Finally, the uniqueness of the $C^*$-inductive limit (\cite{LBR}) implies that ${\mathcal{A}}=\overline{\bigcup\limits_{n<m}A_{-2n,2m}}$
and ${\mathcal{A}}_{(H,G)}=\overline{\bigcup\limits_{n<m}{\mathcal{A}_H}(\Lambda_{-n,m})}$.
As a result, ${\mathcal{A}}$ is  $C^*$-isomorphic to ${\mathcal{A}}_{(H,G)}$.
\end{proof}

\begin{remark}
From Theorem 3.1, one can see that for a normal subgroup $H$ of $G$ the observable algebra related to $H$ in the field algebra ${\mathcal{F}}$ of $G$-spin models ${\mathcal{A}}_{(H,G)}$ can be defined as
\begin{eqnarray*}
    \begin{array}{c}
{\mathcal{A}}_{(H,G)}\ =\ \cdots H\bigotimes_{R_{-2,-1}} {\widehat{G}}\bigotimes_{R_{-1,0}} H\bigotimes_{R_{0,1}}{\widehat{G}}\bigotimes_{R_{1,2}} H\bigotimes_{R_{2,3}}{\widehat{G}}\cdots.
\end{array}
\end{eqnarray*}
In particular, take $G$ as a normal subgroup of $G$, then we have the observable algebra ${\mathcal{A}}_G$ (\cite{F.Nil}) can also be expressed as
\begin{eqnarray*}
    \begin{array}{c}
{\mathcal{A}}_G\ =\ \cdots G\bigotimes_{R_{-2,-1}} {\widehat{G}}\bigotimes_{R_{-1,0}} G\bigotimes_{R_{0,1}}{\widehat{G}}\bigotimes_{R_{1,2}} G\bigotimes_{R_{2,3}}{\widehat{G}}\cdots.
\end{array}
\end{eqnarray*}
What is more, we can get ${\mathcal{A}}_{(H,G)}\subseteq {\mathcal{A}}_G$, from the above expressions of ${\mathcal{A}}_{(H,G)}$ and ${\mathcal{A}}_G$, which is different from ${\mathcal{A}}_G\subseteq {\mathcal{A}}_{(G,H)}$ (\cite{JLN}).
\end{remark}

\begin{remark}
Notice that the linear map
$\varphi\colon G\bigotimes \widehat{H}\rightarrow\widehat{H}$ defined naturally by
\begin{eqnarray*}
    \begin{array}{c}
    \varphi_g(\delta_h)=g\rightarrow\delta_h=\sum\limits_{t\in H}\delta_t\delta_{t^{-1}h}(g)=\left\{\begin{array}{lcl}
                          \delta_{hg^{-1}}, & {\mathrm{if}} &  g\in H \\[3pt]
                          0, & {\mathrm{if}} & g\in G/H
                        \end{array}\right.
\end{array}
\end{eqnarray*}
is not a left action of $G$ on $\widehat{H}$. Thus, ${\widehat{H}}\bigotimes_{R_{1,2}}G$ can not be defined for any subgroup $H$ of $G$, and then
\begin{eqnarray*}
    \begin{array}{c}
\cdots G\bigotimes_{R_{-2,-1}} {\widehat{H}}\bigotimes_{R_{-1,0}} G\bigotimes_{R_{0,1}}{\widehat{H}}\bigotimes_{R_{1,2}} G\bigotimes_{R_{2,3}}{\widehat{H}}\cdots
\end{array}
\end{eqnarray*}
can not be defined naturally. However, the observable algebra ${\mathcal{A}}_{(G,H)}$ in the field algebra ${\mathcal{F}}$ is well defined \cite{JLN}, which can be obtained as the fixed point algebra
\begin{eqnarray*}
    \begin{array}{c}
    {\mathcal{A}}_{(G,H)}\ =\ {\mathcal{F}}^{D(G;H)}\ \equiv\  \{F\in {\mathcal{F}}\colon a(F)=\varepsilon(a)F, \ \forall a\in D(G;H)\}.
    \end{array}
\end{eqnarray*}
Here $D(G;H)$ denotes the crossed product of $C(G)$ and $\Bbb{C}G$ with respect to the adjoint to the action of the latter on the former, $\varepsilon$ is the counit of $D(G;H)$, and ${\mathcal{F}}$ is the field algebra of $G$-spin models.
\end{remark}
{\bf Acknowledgements}\\
This research was supported by National Science Foundation of China (10971011,11371222).

\end{document}